\documentclass[12pt,draft]{amsart}
\usepackage{amssymb,latexsym, amsmath, amscd, array, graphicx}

\usepackage{tikz-cd}

\usepackage{amsthm}
\usepackage{hyperref}
\usepackage[all]{xy}
\usepackage[scr]{rsfso}

\newtheorem{thm}{Theorem}[section]
\newtheorem{theorem}[thm]{Theorem}

\newtheorem{lemma}[thm]{Lemma}

\newtheorem{prop}[thm]{Proposition}

\theoremstyle{definition}

\theoremstyle{definition}

\newtheorem{rem}[thm]{Remark}
\newtheorem{remark[thm]}{Remark}

\newtheorem{cor}[thm]{Corollary}
\newtheorem{con}[thm]{Conjecture}

\theoremstyle{remark}

\DeclareMathOperator{\cat}{{\mbox{\rm cat$_{\rm LS}$}}}

\def\Im{\protect\operatorname{Im}}

\def\m{\medskip}

\long\def\forget#1\forgotten{} %

\numberwithin{equation}{section}

\begin{document}

\title[Rudyak's conjecture for lower dimensional  manifolds]{Rudyak's conjecture for lower dimensional 1-connected manifolds
}
\author[A.~Dranishnikov]{Alexander Dranishnikov}

\author[D. Kundu]{Deep Kundu}

\address{Department of Mathematics, University
of Florida, 358 Little Hall, Gainesville, FL 32611-8105, USA}
\curraddr{}
\email{dranish@math.ufl.edu}
\email{deepkundu@ufl.edu}
%\thanks{Supported by NSF, grant DMS-0904278}

%    \subjclass is required.

\date{}

\dedicatory{In memory of Yuli Rudyak}

%    "Communicated by" -- provide editor's name; required.
%\commby{Daniel Ruberman}

%    Abstract is required.
\begin{abstract}
Rudyak's conjecture states that for any degree one map $f:M\to N$ between oriented closed manifolds there is the inequality  $\cat (M)\ge \cat(N)$
for the Lusternik-Shnirelmann category.
We prove the Rudyak's conjecture for $ n$-dimensional simply connected spin manifolds for $n\le 8$.
\end{abstract}

\maketitle

\section{Introduction} The reduced {\em Lusternik-Schnirelmann category}
(briefly LS-category) $\cat X$ of a topological space $X$ is the
minimal number $n$ such that there is an open cover $\{U_0,\dots,
U_n\}$ of $X$ by $n+1$ contractible in $X$ sets. We note that the LS-category 
is a homotopy invariant. 
Most of deep applications of the LS-category are based on 
the classical Lusternik and Schnirelmann theorem
\cite{CLOT} which states that $\cat M+1$ gives a low bound
for the number of critical points on a  manifold $M$ of any smooth not
necessarily Morse function. The concept of the LS-category and the above theorem were brought by Lusternik and Schnirelmann
to solve the Poincare problem on the existence of three closed geodesics on a 2-sphere for any Riemannian metric on it~\cite{LS1},\cite{LS2}. 
In modern time Yu. Rudyak applied the LS-category in his approach to the Arnold conjecture on the number of fixed points of symplectomorphisms~\cite{Ru2}.

The LS-category is a numerical homotopy invariant which is difficult to compute. Clearly $\cat(S^n)=1$. The converse is also true in view of
the classical results by Smale, Freedman, and Perelman:  If $\cat(M)=1$ then $M$ is homeomorphic to a sphere.
It was proven in~\cite{DKR} that if  $\cat(M)=2$, then the fundamental group of $M$ is free.

Moreover many natural conjectures on the LS-category turned out to be false.
Perhaps the most famous was the Ganea Conjecture of the 60s (see~\cite{Ga}) which states that $\cat(M\times S^n)=\cat M+1$ for a closed manifold $M$.
On the break of the 20th century
N. Iwase constructed two simply-connected 16-dimensional counter-examples $M_2$ and $M_3$ to the Ganea Conjecture. His examples are quite delicate and they are based on some phenomena for primary 2-components and primary 3-components in the unstable homotopy groups of spheres.

In the late 90s Rudyak proposed the following~\cite{Ru1}:
\begin{con}
For a degree one map $f:M\to N$ between closed orientable manifolds, $\cat(M)\ge\cat(N)$.
\end{con}
In~\cite{Ru1} Rudyak proved his conjecture with some restrictions on dimension, connectivity and the LS-category of the domain.
It turned out that his conjecture is quite difficult even in the case of the most simple maps of degree one $f:M\# N\to N$, the collapsing of a summand in the connected sum~\cite{DS}.
Perhaps the most interesting and still a challenging case of Rudyak's conjecture is the case of birational morphism $f:M'\to M$ between nonsingular projective varieties.
We note when $M'$ is a blowup of $M$ at a point the Rudyak's conjecture for $f$ follows from the above formula.

Rudyak's conjecture like many other natural conjectures on the LS-category could be false.
In~\cite{Dr1} there was an attempt to construct a counterexample to the Rudyak conjecture by means of Iwase's manifolds.
Using Iwase's examples the first author
constructed a map of degree one $f:M\to N$ between 32-dimensional  simply-connected manifolds which would be a counter-example to Rudyak's conjecture if $\cat(M_2\times M_3)\ge 5$.
We recall that $\cat(M_2)=\cat(M_3)=3$. The classical inequality~\cite{CLOT}  $$\cat(X\times Y)\le\cat (X)+\cat(Y)$$
and computations in~\cite{Dr1}  imply that $4\le \cat(M_2\times M_3)\le 6$.

A different idea for a counterexample was suggested in~\cite{Dr}.

In any case we believe in Rudyak's conjecture for lower dimensional manifolds. In~\cite{Ru3} Rudyak verified his conjecture for manifolds of dimension $\le 4$.
In this paper we prove Rudyak's conjecture for simply connected spin manifolds of dimension $\le 8$ (Theorem~\ref{main}). At the end of the paper we give an example of non-spin simply connected 7-dimensional manifold. The existence of such an example does not allow us to extend our approach to non-spin manifolds.

\section{Preliminaries}
\subsection{Properties of maps of degree one}
The following is well-known.
\begin{prop}\label{split}
 Let $f \colon M\to N$ be a map between closed connected orientable $n$-manifolds of degree one. Then the map $f_* \colon H_*(M;R) \to H_* (N;R) $ is a
split epimorphism and 
$f^* \colon H^* (N;R)  \to H^*(M;R)$ is a split monomorphism for any coefficient ring. 
 \end{prop} 
\begin{proof}
The spitting map for $f_*:H_*(M;R)\to H_*(N;R)$ is the composition $$H_*(N;R)\stackrel{PD}\to H^{n-*}(N;R)\stackrel{f^*}\to H^{n-*}(M;R) \stackrel{PD}\to H_*(M;R)$$ and for
$f^*:H^*(N;R)\to H^*(M;R)$ is  the composition $$H^*(M;R)\stackrel{PD}\to H_{n-*}(M;R)\stackrel{f_*} \to H_{n-*}(N;R)\stackrel{PD}\to H^*(N;R).$$
\end{proof}

\subsection{G. W. Whitehead Theorem} For nice spaces, such as CW complexes,
it is an easy observation that $\cat X\le\dim X$. 
D.P. Grossmann~\cite{Gro} and independently  G.W.
Whitehead~\cite{Wh},~\cite{CLOT} proved that for simply connected CW
complexes $\cat X\le\dim X/2$. Moreover, the following inequality holds true~\cite{CLOT}.
\begin{thm}\label{WhTh} 
For $k$-connected CW complexes  $$\cat X\le\dim X/(k+1).$$
\end{thm}

\subsection{Ganea's approach to the LS-category}

Recall that an element of an iterated join $X_0*X_1*\cdots*X_n$ of topological spaces is a formal linear combination $t_0x_0+\cdots +t_nx_n$ of points $x_i\in X_i$ with $\sum t_i=1$, $t_i\ge 0$, in which all terms of the form $0x_i$ are dropped. Given fibrations $f_i\colon X_i\to Y$ for $i=0, ..., n$, the fiberwise join of spaces $X_0, ..., X_n$ is defined to be the space
\[
    X_0*_Y\cdots *_YX_n=\{\ t_0x_0+\cdots +t_nx_n\in X_0*\cdots *X_n\ |\ f_0(x_0)=\cdots =f_n(x_n)\ \}
\]
and the fiberwise join of fibrations $f_0, ..., f_n$ is the fibration 
\[
    f_0*_Y*\cdots *_Yf_n\colon X_0*_YX_1*_Y\cdots *_YX_n \longrightarrow Y
\]
defined by taking a point $t_0x_0+\cdots +t_nx_n$ to $f_i(x_i)$ for any $i$ with $t_i>0$. As the name `fiberwise join' suggests, the fiber of the fiberwise join of fibrations is given by the join of fibers of fibrations. 

When $X_i=X$ and $f_i=f:X\to Y$ for all $i$  the fiberwise join of spaces is denoted by $*^{n+1}_YX$ and the fiberwise join of fibrations is denoted by $*_Y^{n+1}f$. 

For a topological space $X$, we turn an inclusion of a point $*\to X$ into a fibration $p_0^X:G_0(X)\to X$. The $n$-th Ganea space of $X$ is defined to be the space $G_n(X)=*_X^{n+1}G_0(X)$, while the $n$-th Ganea fibration $p_n^X:G^n_X\to X$ is the fiberwise join of fibrations $p_0^X:G_0(X)\to X$. 

The following theorem was proven by Schwartz in a more general form~\cite{Sv}:
\begin{thm}\label{Schwartz}  The inequality $\cat(X)\le n$ holds if and only if the fibration $p^X_n:G_n(X)\to X$ admits a section. 
\end{thm}
We note that the fiber of $p_0^X$ is homotopy equivalent to the loop space $\Omega(X)$. Thus, the fiber of $p_n^X$ is homotopy equivalent to the join product $\ast^{n+1}\Omega(X)$.

We recall that the LS-category of a map $f:Y\to X$ is the least integer $k$ such that $Y$ can be covered by
$k+1$ open sets $U_0,\dots, U_k$ such that the restrictions $f|_{U_i}$ are null-homotopic for all $i$.
The above theorem can be extended to maps~\cite{Dr}:
\begin{thm}\label{ganea-maps}
\label{t:ganea}
For a map $f:Y\to X$ to a CW-space~$X$,~$\cat(f)\le n$ if and only if there exists a
lift of $f$ with respect to $p_n^X:G_n(X)\to X$.
\end{thm}
We note that $\cat f\le\min\{\cat X,\cat Y\}$. We recall that the cohomology  cup-length of $X$ a lower bound for $\cat X$.
Similar statement holds true for maps: $\cat f\ge$ cup-length$(\Im f^*)$ (see for example~\cite{Sr}).

\

We call a map $f:X\to Y$ a $k$-equivalence if it induces isomorphism of homotopy groups in dimension $<k$ and an epimorphism in dimension $k$.

The following Proposition was proven in~\cite{DKR} and in~\cite{DS}:	
\begin{prop} \label{p:2} Let $f\colon X\to Y$ be an $(n-2)$-equivalence of $(r-1)$-connected pointed CW-complexes with $r\ge 0$. Then the map $*^{s+1}f$ is a $(sr+s+n-2)$-equivalence. 
\end{prop}

\section{Reduction to the 2 vs 3 case}

\

All manifolds in the paper are considered to be connected, closed, orientable, and smooth.

\begin{prop}\label{le 5}
Let $N$ be a $n$-dimensional simply-connected manifold, $n\le 5$. Let $f:M\to N$ be a map of degree 1. Then $\cat (M) \geq \cat(N)$. 
\end{prop}
\begin{proof}
By Theorem~\ref{WhTh} $\cat N\le 2$. If  $\cat(M)=1$ then $M$ is a homotopy sphere. By Proposition~\ref{split}  $N$ is a 
homology sphere. Since $N$ is simply connected, by the Hurewicz theorem it is a homotopy sphere. Then by the J. H. C. Whitehead
 theorem $N$ is homotopy equivalent to $S^n$. Hence $\cat (N)=1$. So, always $\cat(M)\geq \cat(N)$.
\end{proof}

\begin{prop}\label{1}
Let $f:M\to N$ be a map of degree one between simply connected manifolds of dimension $\le 8$. Suppose that $\cat M\ge 3$. Then $\cat M\ge \cat N$.
\end{prop}
\begin{proof}
By Theorem~\ref{WhTh} the LS-category of simply connected manifolds of dimension $\le 7$ does not exceed 3. So, we have to consider only the case of $n=8$ where $\cat M, \cat N\le 4$.
 Assume that $\cat N=4$.
Then there is no section of the Ganea fibration $p_3:G_3(N)\to N$. The fiber $F$ of this fibration is homotopy equivalent to a 6-connected space $\ast^4\Omega(N)$.
The primary obstruction $\kappa\in H^8(N;\pi_7(F))$ to a section of $p_3$ is the only obstruction. Since the primary obstruction is natural, $f^*(\kappa)$ is the primary
obstruction to the lift of $f$ with respect to $p_3$. By Proposition~\ref{split} $f^*(\kappa)\ne 0$. Hence by Proposition~\ref{ganea-maps} $\cat f>3$.  The inequality 
$\cat f\le\cat M$ completes the proof.
\end{proof}

We recall that if a $n$-manifold $M'$ is obtained from $M$ by a surgery then there is a bordism beetween them, i.e. $(n+1)$-dimenisonal manifold $W$
with boundary $\partial W=M\sqcup M'$. If the surgery was in dimension $k$, then $W$ is homotopy equivalent to $M$ with some $(k+1)$-cells attached.
If $M'$ is obtained by a $k$-surgery from $M$, then $M$ is obtained by $(n-k-1)$-surgery from $M'$.

\begin{prop}\label{2}
(a) Suppose that $M'$ is obtained by a $k$-surgery, $k\le 4$, from a simply connected $n$-manifold $M$ with $\cat M\le 2$. Then $\cat W\le 2$
where $W$ is a bordism of the surgery.

(b) Suppose that $M'$ is obtained from a simply connected manifold $M$ with $\cat M\le 2$ by a surgery in dimension 2. Then $\cat M'\le 2$.
\end{prop}
\begin{proof}
(a) The inequality $\cat M\le 2$ implies that there is a section $\sigma:M\to G_2(M)$ of the Ganea fibration $p_M:G_2(M)\to M$.
The bordism $W$ is homotopy equivalent to the space $$X=M\cup (\cup_\alpha D^{k+1}_\alpha)$$ obtained from $M$ by attaching a $(k+1)$-balls along disjoint  embedded $k$-spheres $S^k_\alpha$. We show that $\cat X\le 2$.
Note that the fiber of the Ganea fibration $p_X:G_2(X)\to X$ is 4-connected. Moreover, for any contractible subset $C\subset X$, the preimage  $p_X^{-1}(C)$ is 4-connected.
Let $p_X^M:G^M_2(X)\to M$ be the restriction of $p_X$ over $M$. Clearly $p_X^M$ contains $p_M$ as a subfibration. Then the restriction of $\sigma$ to $S_\alpha$ is a null-homotopic map
$\sigma|_\alpha:S^k_\alpha\to p_X^{-1}(D^{k+1})$ for all $\alpha$. Let $$\sigma'_\alpha:D^{k+1}_\alpha\to  p_X^{-1}(D^{k+1}_\alpha)$$ be an extension. Then the union $\sigma\cup(\cup_\alpha\sigma'_\alpha$ is a homotopy section of $p_X$. Then
the Homotopy Lifting Property produces an honest section of $p_X$.

(b) We note that $W$ is homotopy equivalent to the space $$X'=M'\cup(\cup_\alpha D^{n-2}_\alpha).$$ Hence the inclusion $j:M'\to X'$ is $(n-3)$-equivalence.
Then $\Omega(j):\Omega(M')\to\Omega(X')$ is an $(n-4)$-equivalence.
By Proposition~\ref{p:2} the induced map $$\ast^3\Omega(f):\ast^3\Omega(M')\to\ast^3\Omega X$$ between the fibers of $G_2(M')$ and $G_2(X')$ is a $(2+2+n-4)$-equivalence. This implies that the map $q:G_2(M')\to Z$ to the pull-back
in
\[
\xymatrix{
G_2(M') \ar[r]^q \ar[rd]^{p_{M'}} & Z \ar[d]^{p'} \ar[r]^{j'} & G_2(X) \ar[d]^{p_{X'}}\\
&  M' \ar[r]^j &X'}
\]
is a $n$-equivalence. A section $\sigma:X'\to G_2(X')$ defines a section $s:M'\to Z$. Since the map $q$ is a $n$-equivalence, the map $s$ can be lifted with respect to $q$. Thus,  $p_{M'}$ admits a section and, hence, by Theorem~\ref{Schwartz} $\cat M'\le 2$.
\end{proof}

\

\section{Main result}

The following theorem was proven by I. M. James~\cite{Ja}:
\begin{thm}\label{James}
Let $p\geq2$ and $X$ be a finite $(p-1)$-connected CW complex of dimension $N\leq p(n+1)$. Denote by $\alpha_X \in H^p(X;\pi_p(X))$ the fundamental class of $X$. Then $\cat(X)=n+1$ if and only if $$0\neq (\alpha_X)^{n+1} \in H^{p(n+1)}(X;\otimes^{n+1}\pi_p(X)),$$ where $(\alpha_X)^{n+1}$ is the $(n+1)$-fold cup product of $\alpha_X$.
\end{thm}
\begin{proof}
    For the proof see \cite[Proposition 2.50]{CLOT}.
\end{proof}

\begin{thm}
Let $N$ be a $3k$-dimensional $(k-1)$connected manifold. Let $f:M\to N$ be a map of degree 1. Then $\cat (M) \geq \cat(N)$. 
\end{thm}

\begin{proof}
By Theorem~\ref{WhTh} $\cat(N)\leq 3$. Let us assume that $\cat(N)=2$. If $\cat(M)=1$ then $M$ is a homotopy sphere, so $N$ is a homotopy sphere in view of $\deg f=1$, hence $\cat (N)=1$, which is a contradiction.

  \m Now let $\cat N=3$. The case $\cat (M)=1$ is impossible by previous argument. Suppose that $\cat(M)=2$. Let us  assume that $\pi_k(N)\neq 0$. Let $u\in H^k(N;\pi_k(N))$ be the fundamental class. Since $\dim =3k$, we conclude that $u^3\neq 0$, by the James Theorem.  Since $f^*: N \to M$ is a monomorphism (see Proposition~\ref{split}) by previous proposition, we conclude that $(f^*(u))^3\neq 0$.\\
 \end{proof}

 \begin{cor}\label{=6}
Let $N$ be a $6$-dimensional simply-connected manifold. Let $f:M\to N$ be a map of degree 1. Then $\cat (M) \geq \cat(N)$.
 \end{cor}

\

\begin{lemma}\label{main lemma}
     Let $f:M\to N$ be a map of degree 1
between $n$-dimensional simply connected manifolds where $n=7$ or $8$ that induces an isomorphism  of homotopy groups $f_*:\pi_i(M)\to \pi_i(N)$ for $i\le n-5$. Then $\max\{\cat (f),2\} \geq \cat(N)$.
\end{lemma}
\begin{proof}
We may assume that $\cat N=3$. If $n=7$ this is a maximal value for $\cat N$. If $n=8$ and $\cat N=4$, by Theorem~\ref{James} the cup-length of $N$ is 4. Then
by Proposition~\ref{split} the cup-length of the image $\Im f^*$ is 4. Since $\cat f\ge$ cup-length$(\Im f^*)$, we may assume that $\cat f\le 2$.

    Now, since $M$ and $N$ are simply connected, we have $\pi_2(M)$ isomorphic to $H_2(M)$ and  $\pi_2(N)$ is isomorphic to $H_2(N)$ by the Hurewicz theorem. Hence, the map $f$ induces an isomorphism $f_*:H_2(M)\to H_2(N)$.
    
Since $\cat(f)\le 2$, By Theorem~\ref{ganea-maps} there is a lift $\phi:M\to G_2(N)$  of $f$ with respect to the Ganea fibration $p_2^N$,
\[\begin{tikzcd}
	&& {G_2(N)} \\
	\\
	{M} && {N}
	\arrow["{p_2^N}", from=1-3, to=3-3]
	\arrow["\phi", from=3-1, to=1-3]
	\arrow["f"', from=3-1, to=3-3].
\end{tikzcd}\]
We show that there is a section $s:N\to G_2(N)$.

Since $M$ and $N$ are simply connected, the condition of the lemma imply that  $f_*:H_i(M)\to H_i(N)=\pi_i(N)$ are isomorphisms for $i\le n-5$ by the classic Whitehead theorem. Then by the Poincare duality, $f^*:H^i(N)\to H^i(M)$ are isomorphisms for $i\ge 5$.

Note that the fiber $F=\ast^3\Omega N$  of the fibration $p=p_2^N:G_2(N)\to N$ is  is the join product of 3 copies of connected space $\Omega N$
and hence it is $4$-connected. Therefore in the Moore-Postnikov tower for $p:G_2(N)\to N$ 
we have $N=Z_1=Z_2=Z_3=Z_4=Z_5$:

% https://q.uiver.app/#q=WzAsOCxbMCw1LCJHXzIoTikiXSxbMiw1LCJOXjUiXSxbMSw1LCJaXzEiXSxbMSw0LCJaXzIiXSxbMSwzLCJaXzMiXSxbMSwyLCJaXzQiXSxbMSwxLCJaXzUiXSxbMSwwLCJaXzYiXSxbMCwyXSxbMiwxXSxbMCwzXSxbMCw0XSxbMCw1XSxbMCw2XSxbMCw3XSxbNywxXSxbNiwxXSxbNSwxXSxbNCwxXSxbMywxXSxbNyw2XSxbNiw1XSxbNSw0XSxbNCwzXSxbMywyXV0=
\[\xymatrix{
& Z_8\ar[d]^{q_7} \\
           & {Z_7}\ar[d]^{q_6}  \ar[r]^-{\kappa_7} & K(\pi_7(F),8)\\
	{G_2(N)}\ar[ruu]^{p_8} \ar[ru]_{p_7} \ar[r]^{p_6} \ar[rd]^p &  {Z_6} \ar[d]^{q_5} \ar[r]^-{\kappa_6} & K(\pi_6(F),7)\\
           {M} \ar[u]^\phi \ar[r]^f & {N}\ar[r]^-{\kappa_5} & K(\pi_5(F),6)}
	\]
We recall that the maps $q_k:Z_{k+1}\to Z_k$ are principal bundles with fibers $K(\pi_k(F),k)$. 
Thus, there are pull-back diagrams

$$\begin{CD}
Z_{k+1} @>\kappa_k'>> E_{k+1}\\
@Vq_kVV @V\bar p_kVV\\
Z_k @ >\kappa_k>> K(\pi_k(F),k+1)\\
\end{CD}
$$
where  $E_n$ are contractible and $\bar p_n$ are principal bundle with fibers topological abelian groups having homotopy type of the Eilenberg-Maclane spaces  $K(\pi_{n-1}(F),n-1)$
(see~\cite{Mc} or~\cite{Dr0}).

By commutativity of the above diagrams we obtain $$\kappa_5f=\kappa_5q_5p_6\phi=\bar p_5\kappa_5'p_6\phi.$$ 
Since the map $\bar p_5\kappa_5'p_6\phi$ factors through a contractible space $E_6$, the map $\kappa_5f$ is null-homotopic. We recall that $f^*$ is injective on the 5-cohomology.
Since the map $\kappa_5f$ represents the image $f^*[\kappa_5]$ of the cohomology class represented by 
$\kappa_5:N\to K(\pi_5(F),6)$, the map $\kappa_5$
is null-homotopic. Therefore, $\kappa_5$ admits a lift and, hence,  $q_5$ admits a section. Since $q_5$ is a principal bundle that has a section, it is a trivial bundle over $N$ with the fiber $K(\pi_5(F),5)$.

Thus, $$Z_6\cong  K(\pi_5(F),5)\times N.$$ The projection $pr:Z_6\to K(\pi_5(F),5)$ to the fiber defines a cohomology class $$\alpha=[pr\circ p_6\circ\phi]\in H^5(M;\pi_5(F)).$$
Since $f^*$ is an isomorphism of 5-cohomology, there is $\beta\in H^5(N;\pi_5(F))$ with $f^*(\beta)=\alpha$. The class $\beta$ can be represented by a map
$$s:N\to K(\pi_5(F),5)$$ which defines a section $s_5:N\to Z_6$ of $q_5$ such that the map $$p_6\phi=(pr\circ p_6\circ\phi,f):M\to K(\pi_5(F),5)\times N$$ is homotopic to $$s_5 f=(sf,f):M\to K(\pi_5(F),5)\times N.$$

Next we show that the map $s_5:N\to Z_6$ admits a lift with respect to $q_6$. This is equivalent to say that $$\kappa_6s_5:N\to K(\pi_6(F),7)$$ is null-homotopic.
Since $p_7\phi$ is a lift of the map $p_6\phi$ with respect to $q_6$, the composition 
$$
\kappa_6p_6\phi:M\to K(\pi_6(F),7)
$$
is null-homotopic. Therefore, $\kappa_6s_5f$ is null-homotopic. This means that $f^*$ takes the cohomology class $[\kappa_6s_5]$ to zero. Since by proposition~\ref{split} the homomorphism $f^*$ is injective, we obtain that $[\kappa_6s_5]=0$. Thus, the map $\kappa_6s_5$ is null-homotopic.

Suppose that $n=7$.
Since the map $p_7$ is 7-equivalence i.e. it is an isomorphism for $i$-dimensional homotopy groups for $i<7$ and epimorphism for $i=7$, there is a homotopy lift
$\sigma :N\to G_2(N)$ of the map $s_6:N\to Z_7$ with respect to $p_7$. Then $\sigma $ is a homotopy section of $p$. Since $p$ is a fibration, it admits a regular section.
Therefore, $\cat N\le 2$.

If $n=8$, we construct a lift $s_7:N\to Z_8$ of $s_5$ with respect to $q_6q_7$ and complete the proof similarly to the case of $n=7$. 
For that we modify $s_6$ to a such a map that $s_6f$ becomes homotopic to $p_7\phi$.

Since $q_6s_6=s_5$ is an embedding, the map $q_6$ is defined by a section of $q_6$ on $s_5(N)$.
Since  the map $$\kappa_6s_5:N\to K(\pi_7(F),7)$$ is null-homotopic,
the restriction of $q_6$ over $s_5(N)$ is a trivial bundle, $$q_6^{-1}(s_5(N))\cong s_5(N)\times K(\pi_6(F),6).$$
Since $p_6\phi$ is homotopic to $s_5f$ and $q_6p_7\phi=p_6\phi$, by the Homotopy Lifting property of $q_6$, we can homotop the map $p_7\phi$ to a map $\psi:M\to Z_7$ such that $q_6\psi=s_5f$. Thus the map $$\psi:M\to s_5(N)\times K(\pi_6(F),6)$$
has two components $(s_5f,\psi_1)$. Since $f$ induces an isomorphism of 6-dimensional cohomology groups, there is a map $$\xi:N\to K(\pi_6(F),6)$$ such that $f^*$ takes the cohomology class $[\xi]\in H^6(N;\pi_6(F))$ to the cohomology class $[\psi_1]\in H^6(M;\pi_6(F))$.
Thus, $\psi_1$ is homotopic to $\xi f$. We consider a section $$s':s_5(N)\cong N\to q_6^{-1}(s_5(N))\cong s_5(N)\times K(\pi_6(F),6)$$ defined by $\xi$. Then $\psi$ is homotopic to the map $s's_5f$.
Since $p_8\phi$ is a lift of the map $p_7\phi$ with respect to $q_7$, the composition 
$$
\kappa_7p_7\phi:M\to K(\pi_6(F),7)
$$
is null-homotopic. Hence $\kappa_7\psi$ is null-homotopic and the map $\kappa_7 s's_5f$ is null-homotopic. 
This means that $f^*$ takes the cohomology class $[\kappa_7 s's_5]$ to zero. Since the homomorphism $f^*$ is injective, we obtain that $[\kappa_7s's_5]=0$. Thus, the map $\kappa_7s's_5$ is null-homotopic and, hence, there is a lift $s_7:N\to Z_8$ of $s's_5$ with respect to $q_7$. 

Since the map $p_8$ is 8-equivalence, there is a homotopy lift
$\sigma :N\to G_2(N)$ of the map $s_7:N\to Z_8$ with respect to $p_8$. Then $\sigma $ is a homotopy section of $p$. Since $p$ is a fibration, it admits a regular section.
Therefore, $\cat N\le 2$.
\end{proof}

\

\begin{prop}\label{surgery}
A degree one map $f:M\to N$ of a $n$-dimensional, $n=7,8$ simply connected spin manifold  with $\cat M\le 2$ is bordant by means of surgery in dimension 2 and 3 to a a degree one map $f':M'\to N$  with $\cat f'\le 2$ that induces isomorphisms $f':\pi_i(M')\to\pi_i(N)$ for $i\le n-5$.
\end{prop}
\begin{proof}
For $n=7,8$ by Proposition~\ref{split} the induced homomorphism $f_*:H_2(M)\to H_2(N)$ is surjective with  a finitely generated kernel $K=ker f_*$ which is by
the Hurewicz Theorem  isomorphic to $ker\{f_*:\pi_2(M)\to\pi_2(N)\}$. Let $S_1,\dots, S_k$ be disjoint 2-spheres smoothly embedded in $M$ that
generate $K$. For each sphere $S_i$ its normal bundle $\nu_i$ is classified by a map $\nu_i:S_i\to BSO(n-2)$. The map $\nu_i$  is null-homotopic, since otherwise the evaluation of the Stiefel-Whitney class $w_2\in H^2(BSO(n-2);\mathbb Z_2)$ on the image $(\nu_i)_*([S_i])\in H_2(BSO(n-2);\mathbb Z_2)$ would be nonzero.
The latter  would contradict to the spin condition on $M$. Thus, we can perform a surgery to obtain a bordism $W$ of $M$ to $M'$ together with a map $\bar f:W\to N$.
The restriction  $f':M'\to N$ of $\bar f$ to $M'$ has degree one and $f'$ induces an isomorphism of 2-dimensional homotopy groups. For $n=7$ by Proposition~\ref{2} (b)
$\cat M'\le 2$ and, hence, $\cat f'\le 2$.

When $n=8$ by the above we may assume that $f:M\to N$ induces an isomorphism of 2-dimensional homotopy groups.
Since $\pi_3(BSO(n-3))=0$, we obtain that the normal bundle of any embedded 3-sphere in $M'$ is trivial.
Hence we can perform surgery on generators of  the kernel $ker\{f_*:\pi_3(M')\to\pi_3(N)\}$ to get a bordism $W$ between $M$ and $M'$ together with a map $\hat f:W\to N$ that agrees with $ f$ on $M$ such that the restriction $f':M'\to N$ of $\hat f$ to $M'$ has degree one and it induces isomorphisms of homotopy groups of $\dim \le 3$. 
By Proposition~\ref{2} (a) we obtain that $\cat W\le 2$. Hence $\cat\hat f\le 2$. By Theorem~\ref{ganea-maps} $\hat f$ admits a lift $\hat\phi:W\to G_2(N)$ of $\hat f$ with respect to $p_2^N$.
The restriction of $\hat\phi$ to $M'$ is a lift of $f':M'\to N$ with respect to $p^N_2$. Then by Theorem~\ref{ganea-maps} $\cat f'\le 2$.
\end{proof}

\begin{thm}\label{main}
A degree one map $f:M\to N$ between simply connected spin $n$-manifold for $n\le 8$ satisfies the Rudyak conjecture,  $$\cat M\ge \cat N.$$
The spin condition can be dropped when $n\le 6$.
\end{thm}
\begin{proof}
When $n\le 5$ the statement  is proven in Proposition~\ref{le 5}.
In view of Proposition~\ref{1} it suffice to show that there is no degree one map $f:M\to N$ with $\cat M\le 2$ and $\cat N\ge 3$.
The case of $n=6$ is taken care of in Corollary~\ref{=6}.

We consider the case when $n\ge 7$ and $M$ is spin.  Let $\cat M=2$. In view of Proposition~\ref{surgery} and Proposition~\ref{2} we may assume that $f$ 
induces isomorphisms of homotopy groups of dimension $\le 3$ and $\cat f\le 2$.
Then by Lemma~\ref{main lemma} $\cat N\le 2$.
\end{proof}

\

\begin{rem}
There are non-spin 7-manifolds with the LS-category 2.
We recall that the Wu manifold SU(3)/SO(3) is simply connected and non-spin. A 7-manifold $M$ constructed from the product of a 2-sphere and the Wu manifold $$S^2\times (SU(3)/SO(3))$$ by a 2-surgery on the factor $S^2$ is called a
gyration of the Wu manifold~\cite{Xu}. It is simply connected and non-spin.
Below we compute its LS-category.
\end{rem}
\begin{prop}\label{example}
The LS-category of the gyration of the Wu manifold is 2.
\end{prop}
\begin{proof}
Let $Y=SU(3)/SO(3)$. By Theorem~\ref{WhTh} $\cat Y=2$.
The bordism $X$ defined by the 2-surgery on $Y\times S^2$ is homotopy equivalent to the union $(Y\times S^2)\cup(\{y_0\}\times D^3)$ which is homotopy equivalent to the half-smash product $$Y\rtimes S^2=(Y\times S^2)/(\{y_0\}\times D^3).$$ It is known that for the half-smash product~\cite{SS} $$\cat(A\rtimes B)=\cat A.$$ Hence $\cat X=\cat Y=2$.
Let $M$ be the gyration of the Wu manifold.  Then $X$ is homotopy equivalent to $X'=M\cup D^5$.  The rest of the proof coincides with the end of the proof of Proposition~\ref{2} (b).
Since the inclusion $j:M\to X'$ is a 4-equivalence,
$\Omega(j):\Omega(M)\to\Omega(X')$ is a 3-equivalence.
By Proposition~\ref{p:2} $$\ast^3\Omega(j):\ast^3\Omega(M)\to\ast^3\Omega (X')$$  is a $(2+2+3=7)$-equivalence. This implies that the map $q:G_2(M)\to Z$ to the pull-back
in the diagram
\[
\xymatrix{
G_2(M) \ar[r]^q \ar[rd]^{p_{M}} & Z \ar[d]^{p'} \ar[r]^{j'} & G_2(X') \ar[d]^{p_{X'}}\\
&  M \ar[r]^j &X'}
\]
is a $7$-equivalence. Since $\cat X'=\cat X=2$, there is a section $\sigma:X'\to G_2(X')$ which defines a section $s:M\to Z$ of $p'$. Since the map $q$ is a $7$-equivalence,   $s$ can be lifted with respect to $q$. Thus,  $p_{M}$ admits a section and, hence, $\cat M\le 2$. 

Since  $M$ is not homeomorphic to a sphere, $\cat M>1$ and we obtain that $\cat M= 2$. 
\end{proof}

\


\begin{thebibliography}{[CLOT]}
\bibliographystyle{amsplain}


\bibitem{Bro}Browder, W. Surgery on Simply-Connected Manifolds, Springer, 1972.


\bibitem{CLOT} Cornea, O.; Lupton, G.; Oprea, J.; Tanr\'e, D.:
Lusternik-Schnirelmann category.  {\em Mathematical Surveys and
	Monographs}, \textbf{103}.  American Mathematical Society, Providence,
RI, 2003.

\bibitem{Dr0} A. Dranishnikov, Free abelian topological groups and collapsing maps, Topology Appl. 159 (2012), no. 9, 2353-2356.

\bibitem{Dr} A. Dranishnikov, The LS-category of the product of lens spaces, Algebr. Geom. Topol. 15 (2015) no. 5, 2985-3010.

\bibitem{Dr1} A. Dranishnikov,  On the LS-category of product of Iwase's manifolds, Proc. Amer. Math. Soc.  150  (2022),  no. 5, 2209-2222.


\bibitem{DKR} A. Dranishnikov, M. Katz, Yu. Rudyak,
Small values of Lusternik-Schnirelmann category for manifolds.  {Geometry and Topology}, 12 (2008) no. 3,
1711-1727.



\bibitem{DS}
A. Dranishnikov, R. Sadykov,
{\em The Lusternik-Schnirelmann category of a connected sum}, Fundamenta Mathematicae, 251 (2020), 317-328.




\bibitem{Ga} T. Ganea, \emph{Lusternik-Schnirelmann category and strong category}, Illinois J. Math. 11 (1967), 417-427.


\bibitem{Gro} Grossmann, D.P.: On the estimate for the category of
Lusternik-Schnirelmann, {\em Dokl. Akad. Nauk SSSR} (1946), 109-112.

\bibitem{Ha} A. Hatcher, Algebraic Topology,  Cambridge University Press, 2002.

\bibitem{Iw1} N.~Iwase \emph{Ganea's Conjecture on Lusternik-Schnirelmann category}, Bull. Lond. Math. Soc. 30, (1998), 623-634.


\bibitem{Iw2} N.~Iwase, \emph{Lusternik-Schnirelmann category of a sphere-bundle over a
sphere},  Topology 42 (2003) 701-713.

\bibitem{Ja} James I. M. On category in the sense of Lusternik-Schnirelmann, Topology 17 (1978) 331-348.

\bibitem{LS1} Lusternik, L., Schnirelmann, L. Sur le probleme de trois geodesiquesfermees sur les surfaces de genus 0, Comptes Rendus de l'Academie Sciences de Paris, 189 (1929) 269-271.

\bibitem{LS2}
Lusternik, L. and Schnirelmann, L. 
Methodes Topologiques dans les problemes variationnels.
{\em Hermann,} Paris, 1934. 

\bibitem{Mc} M. C. McCord, Classifying spaces and inﬁnite symmetric products, Transactions of AMS, 146 (1969), 273-298.

\bibitem{MS} Milnor J., Stasheff J. Characteristic Classes, Princeton University Press, 1974.

%\bibitem{Ru} Rudyak Yu., {\em On Thom spectra, orientability and cobordism.} Springer, 1998.

\bibitem{Ru1} Rudyak, Yu. On category weight and its applications, Topology 38 No 1 (1999), 37-55.

\bibitem{Ru2} Rudyak, Yu., On analytical applications of stable homotopy (the Arnold conjecture, critical points), Math. Z. 230 (1999), 659-672.

\bibitem{Ru3} Rudyak, Yu. Maps of degree 1 and Lusternik-Schnirelmann category,  Topology Appl. 221 (2017), 225–230. 

\bibitem{Sr} Srinivasan, T.  The Lusternik-Schnirelmann category of metric spaces, Topology Appl. 167 (2014), 87-95.
 
\bibitem{SS} Stanley D., Strom J., Lusternik-Schnirelmann category of product with half-smashes, AGT 20 (2020), 439-450.



 \bibitem{Sv} A.\v Svarc, The genus of a fibered space. {Trudy
Moskov. Mat. Ob\v s\v c. }\textbf {10, 11} (1961 and 1962), 217--272, 99--126,
(in {\em Amer. Math. Soc. Transl.} Series 2, vol \textbf{55} (1966)).


\bibitem{Wh} Whitehead, G.W.: The homotopy suspension {\em
Colloque de Topologie algebrique tenu \'a Louvain} (1956) 89-95.

\bibitem{Xu} Fupeng Xu, On the simplest simply connected non-spin rational homology
7-spheres that are not 2-connected, arXiv:2603.18661v1
\end{thebibliography}
\end{document}